\documentclass[12pt]{article}

\usepackage{amsmath, amsfonts,amssymb,  amscd, amsthm, verbatim, mathtools, hyperref} 
\usepackage{fullpage}
\usepackage{verbatim, colonequals}
\usepackage[all]{xy}
\usepackage{tikz,tkz-euclide}

\DeclareMathOperator{\Prob}{Prob}

\def\A{\mathbb{A}}

\def\P{\mathbb{P}}

\def\F{\mathbb{F}}

\theoremstyle{plain}

\newtheorem{thm}{Theorem}
\newtheorem{Lem}[thm]{Lemma}
\newtheorem{Cor}[thm]{Corollary}

\newtheorem{Conj}[thm]{Conjecture}

\theoremstyle{remark}
\newtheorem{Rem}[thm]{Remark}

\newtheorem{Exa}[thm]{Example}

\renewcommand\footnotemark{}

\makeatletter
\newcommand{\subjclass}[2][2020]{%
  \let\@oldtitle\@title%
  \gdef\@title{\@oldtitle\footnotetext{#1 \emph{Mathematics subject classification.} #2}}%
}
\newcommand{\keyywords}[1]{%
  \let\@@oldtitle\@title%
  \gdef\@title{\@@oldtitle\footnotetext{\emph{Key words and phrases.} #1.}}%
}
\makeatother

\begin{document}

\title{Improved Lang--Weil bounds for a geometrically irreducible hypersurface over a finite field}
\keyywords{Lang--Weil bound, hypersurface, Bertini's theorem, random sampling}
\author{Kaloyan Slavov
\thanks{This research was
supported by NCCR SwissMAP of the SNSF.}
}

\maketitle

\begin{abstract}
We sharpen to nearly optimal the known asymptotic and explicit bounds for the number of 
$\mathbb{F}_q$-rational points on a geometrically irreducible hypersurface over a (large) finite field. 
The proof involves a Bertini-type probabilistic combinatorial technique. Namely, 
we study the number of $\mathbb{F}_q$-points on the intersection of the given hypersurface with a random plane.
\end{abstract}

\section{Introduction}
\label{sec:Intro}

Let $n\geq 2$, $d\geq 1$, and let $\F_q$ be a finite field.
Let $X\subset\A^n_{\F_q}$ be a geometrically irreducible hypersurface of degree $d$. Lang and Weil \cite{Lang_Weil} have established the bound
\begin{equation}
||X(\F_q)|-q^{n-1}|\leq (d-1)(d-2)q^{n-3/2}+C_d q^{n-2},
\label{eq:Lang_Weil_classical}
\end{equation}
where $C_d$ depends only on $d$ and $n$.

We summarize the smallest known possible values of $C_d$ available in the literature.
\begin{description}
\item[a)] Suppose that $n=2$. Aubry and Perret \cite{Aubry_Perret} prove that 
\begin{equation}
q-(d-1)(d-2)\sqrt{q}-d+1\leq |X(\F_q)|
\leq q+(d-1)(d-2)\sqrt{q}+1.
\label{eq:Aubry_Perret}
\end{equation}

\item[b)] Ghorpade and Lachaud \cite{GL} prove that one can take $C_d=12(d+3)^{n+1}$ in \eqref{eq:Lang_Weil_classical}.
\item[c)] Cafure and Matera \cite{Cafure_Matera} prove that one can take $C_d=5d^{13/3}$ in \eqref{eq:Lang_Weil_classical}; moreover, if $q>15d^{13/3}$, one can take $C_d=5d^2+d+1$. 
\item[d)] The author \cite{Slavov_FFA} has established 
the lower bound (for any $\varepsilon>0$)
\[|X(\F_q)|\geq q^{n-1}-(d-1)(d-2)q^{n-3/2}-(d+2+\varepsilon)q^{n-2}\]
for $q\gg 1$ (with an explicit condition on $q$).

\item[e)] The author's Theorem 8 in the preprint \cite{Slavov_preprint} implies that for every $\varepsilon>0$, $\varepsilon'>0$, we have
\[|X(\F_q)|\leq q^{n-1}+(d-1)(d-2)q^{n-3/2}+((2+\varepsilon)d+1+\varepsilon')q^{n-2}\] 
as long as $q\gg 1$ (again with an explicit condition on 
$q$).    
\end{description} 

In this note we tighten the known asymptotic and explicit bounds for $|X(\F_q)|$ when $q$ is large relative to $d$. 

We first look at upper bounds. 

\begin{thm}
Let $X\subset\A^n_{\F_q}$ be a geometrically irreducible hypersurface of degree $d$. Then
\begin{equation}
|X(\F_q)|\leq q^{n-1}+(d-1)(d-2)q^{n-3/2}+\left(1+\pi^2/6\right)q^{n-2}+O_{d}(q^{n-5/2}),
\label{eq:pi}
\end{equation}
where the implied constant depends only on $d$ and can be computed effectively. 
\label{thm:pi}
\end{thm}

We can exhibit explicit bounds, as in the theorem below. 

\begin{thm}
Let $X\subset\A^n_{\F_q}$ be a geometrically irreducible hypersurface of degree $d$.  
Suppose that $q>15d^{13/3}$. Then
\begin{equation}
|X(\F_q)|\leq q^{n-1}+(d-1)(d-2)q^{n-3/2}+5q^{n-2}.
\label{eq:explicit}
\end{equation}
\label{thm:explicit}
\end{thm}

\begin{Exa}[Cylinder over a maximal curve]
Let $d\geq 3$ be such that $d-1$ is a prime power. Let $q$ be an odd power of $(d-1)^2$.
Consider the  curve $C=\{y^{d-1}+y=x^d\}$ in $\A^2_{\F_q}$. It is known (see, for example, \cite{Stichtenoth}) that 
$\#C(\F_q)=q+(d-1)(d-2)\sqrt{q}$. Thus the 
number of $\F_q$-points on $C\times\A^{n-2}$ is $q^{n-1}+(d-1)(d-2)q^{n-3/2}$.
\label{Exa:Hermitian}
\end{Exa}

\begin{Rem}
While the cylinder $C\times\A^{n-2}$ in Example \ref{Exa:Hermitian} is nonsingular, its Zariski closure in 
$\P^n$ has a large (in fact, $(n-3)$-dimensional) singular locus. In general, let $X\subset\A^n$ be a geometrically irreducible hypersurface such that 
$\#X(\F_q)\geq q^{n-1}+(d-1)(d-2)q^{n-3/2}-O_d(q^{n-2})$ for large $q$. Theorem 6.1 in \cite{GL} implies that the Zariski closure of $X$ in $\P^n$ must have singular locus of dimension $n-3$ or $n-2$.   
\end{Rem}

As in Theorem 4 in \cite{Slavov_FFA}, we can exhibit a forbidden interval for $|X(\F_q)|$. Notice that $X$ is not necessarily geometrically irreducible in the statement below. 

\begin{thm}
Let $X\subset\A^n_{\F_q}$ be a hypersurface of degree $d$. If
\begin{equation}
|X(\F_q)|\leq \frac{3}{2}q^{n-1}-(d-1)(d-2)q^{n-3/2}-
(d^2+d+1)q^{n-2},
\label{eq:N_with_three_halves}
\end{equation}
then in fact
\begin{equation}
|X(\F_q)|\leq q^{n-1}+(d-1)(d-2)q^{n-3/2}+12q^{n-2}.
\label{eq:zone_11}
\end{equation}
\label{thm:zone}
\end{thm}

\begin{Rem}
Let us write $g(d)+\cdots$ for an effectively computable $g(d)+g_1(d)$, where $g_1(d)=o(g(d))$ for 
$d\to\infty$. Theorem \ref{thm:zone} has content when the right-hand side of \eqref{eq:N_with_three_halves} exceeds the right-hand side of \eqref{eq:zone_11}, which takes place for $q>16d^4+\cdots$. Thus in the presence of Theorem \ref{thm:explicit}, Theorem \ref{thm:zone} addresses the range $16d^4+\cdots<q<15d^{13/3}$. Notice that in the Lang--Weil bound \eqref{eq:Lang_Weil_classical}, the approximation term $q^{n-1}$ dominates the error precisely when $q>d^4+\cdots$.  This is why it is reasonable to frame the entire discussion of the Lang--Weil bound in the range $q>d^4+\cdots$. For example, any lower Lang--Weil bound is trivial for $q$ below this threshold.
\end{Rem}

We improve the lower bounds for $|X(\F_q)|$ as well. The proof of Theorem 4 in \cite{Slavov_FFA} actually gives a lower bound which is tighter for $q\gg 1$ than the one stated in \cite{Slavov_FFA}.  

\begin{thm}
Let $X\subset\A^n_{\F_q}$ be a geometrically irreducible hypersurface of degree $d$. Then
\begin{equation}
|X(\F_q)|\geq q^{n-1}-(d-1)(d-2)q^{n-3/2}-dq^{n-2}-O_d(q^{n-5/2}),
\label{eq:lower_bound}
\end{equation}
where the implied constant depends only on $d$ and can be computed explicitly. 
\label{thm:lower_bound}
\end{thm}

We give a version with an explicit lower bound as well.

\begin{thm}
Let $X\subset\A^n_{\F_q}$ be a geometrically irreducible hypersurface of degree $d$. Suppose that $q>15d^{13/3}$. Then
\begin{equation}
|X(\F_q)|\geq q^{n-1}-(d-1)(d-2)q^{n-3/2}-(d+0.6)q^{n-2}.
\end{equation}
\label{thm:explicit_lower_bound}
\end{thm}

\begin{Exa}
As in Example \ref{Exa:Hermitian}, let $d\geq 3$ be such that $q_0\colonequals d-1$ is a prime power. The curve 
$\{y^{d-1}z+yz^{d-1}=x^d\}$ in $\P^2$ over 
$\F_{q_0}$ intersects the line $x=0$ at $d$ distinct points defined over an extension $\F_{q_1}$ of $\F_{q_0}$. Let $q$ be an even power of $q_1$. Then the affine curve 
$C\colonequals \{y^{d-1}z+yz^{d-1}=1\}$ in $\A^2_{\F_q}$ satisfies $\#C(\F_q)=q-(d-1)(d-2)\sqrt{q}-d+1$. Consequently, the number of $\F_q$-points on the hypersurface $C\times\A^{n-2}$ in $\A^n$ is
$q^{n-1}-(d-1)(d-2)q^{n-3/2}-(d-1)q^{n-2}$.
\label{Exa:lower_affine}
\end{Exa}

In fact, the proofs of Theorems \ref{thm:pi} and \ref{thm:lower_bound} give an algorithm that takes as input a half-integer $r\geq 0$ and
constants\footnote{We refer to $C_d^{(j)}$ and $D_d^{(j)}$ interchangeably as constants or as functions of $d$ depending on the context.} $C_d^{(j)}$ and 
$D_d^{(j)}$
for each half-integer $1/2\leq j\leq r$ such that 
\[|X(\F_q)|\leq q^{n-1}+\sum_{j=1/2}^r C_d^{(j)} q^{n-1-j}+O_d(q^{n-r-3/2})
\quad\qquad\text{(summation over half-integers)}\] 
and
\[|X(\F_q)|\geq q^{n-1}-\sum_{j=1/2}^r D_d^{(j)} 
q^{n-1-j}-O_d(q^{n-r-3/2})
\quad\qquad\text{(summation over half-integers)},\]
and returns as output four additional $C_d^{(r+1/2)}$, $C_d^{(r+1)}$, $D_d^{(r+1/2)}$, and $D_d^{(r+1)}$  such that 
\[|X(\F_q)|\leq q^{n-1}+\sum_{j=1/2}^{r+1} C_d^{(j)} q^{n-1-j}+O_d(q^{n-r-5/2})
\quad\qquad\text{(summation over half-integers)}\]
and
\[|X(\F_q)|\geq q^{n-1}-\sum_{j=1/2}^{r+1} D_d^{(j)} q^{n-1-j}-O_d(q^{n-r-5/2})
\quad\qquad\text{(summation over half-integers)}.\]
Initiating the algorithm with $r=0$ and the rather weak version
\[q^{n-1}-O_d(q^{n-3/2})\leq |X(\F_q)|\leq q^{n-1}+O_d(q^{n-3/2})\] of \eqref{eq:Lang_Weil_classical}, we obtain \eqref{eq:pi} and \eqref{eq:lower_bound}.  
In turn, taking \eqref{eq:pi} and \eqref{eq:lower_bound} as input, we obtain

\begin{Cor}
Let $X\subset\A^n_{\F_q}$ be a geometrically irreducible hypersurface of degree $d$. Then
\begin{multline}
|X(\F_q)|\geq q^{n-1}-(d-1)(d-2)q^{n-3/2}-dq^{n-2}
-2(d-1)(d-2)q^{n-5/2}\\
-(2(d-1)^2(d-2)^2+d^2/2+d+2+\pi^2/6)q^{n-3}-O_d(q^{n-7/2}).
\label{eq:long_lower_bound}
\end{multline}
\label{Cor:long}
\end{Cor}

A lower Lang--Weil bound can be useful in proving
that a geometrically irreducible hypersurface 
$X\subset\A^n_{\F_q}$ has an $\F_q$-rational point. It is known (see Theorem 5.4 in \cite{Cafure_Matera} and its proof) that if $q>1.5d^4+\cdots$, then 
$X(\F_q)\neq\emptyset$. Notice that the approximation term 
$q^{n-1}$
in \eqref{eq:long_lower_bound} dominates the remaining explicit terms already for $q>d^4+\cdots$. Based on this heuristic, we state 

\begin{Conj}
There exists an effectively computable function $g_1(d)=o(d^4)$ as $d\to\infty$ with the following property. 
Let $X\subset\A^n_{\F_q}$ be a geometrically irreducible hypersurface of degree $d$. Then $X(\F_q)\neq\emptyset$ as long as $q>d^4+g_1(d)$. 
\end{Conj}

In contrast to the upper bounds, all lower bounds above 
(including \eqref{eq:Aubry_Perret} and Example \ref{Exa:lower_affine}) contain a $d$ in the coefficient of $q^{n-2}$. This discrepancy disappears if we work in projective space. 

\begin{thm}
Let $X\subset\P^n_{\F_q}$ be a geometrically irreducible hypersurface of degree $d$. Then
\begin{align*}
|X(\F_q)|&\geq q^{n-1}-(d-1)(d-2)q^{n-3/2}-O_d(q^{n-5/2})
\qquad\text{and}\\
|X(\F_q)|&\leq q^{n-1}+(d-1)(d-2)q^{n-3/2} +(1+\pi^2/6)q^{n-2}+O_d(q^{n-5/2}).
\end{align*}
\label{thm:projective}
\end{thm}

\begin{Exa}[Cone over a maximal curve]
Let $(d,q_0)$ be such that there exists a (nonsingular) maximal curve $C=\{f=0\}$ in $\P^2$ over $\F_{q_0}$ of degree $d$. Let $q$ be a power of $q_0$ and let $X=\{f=0\}\subset\P^n_{\F_q}$ be a projective cone over $C$. Then
\[\#X(\F_q)=q^{n-1}\pm (d-1)(d-2)q^{n-3/2}+q^{n-2}+q^{n-3}+\cdots+1,\] with $\pm$ depending on whether $q$ is an odd or an even power of $q_0$.
Thus the gap between what is achieved in this example and what is established in Theorem \ref{thm:projective} is $q^{n-2}+O_d(q^{n-5/2})$ in the case of the lower bound and $(\pi^2/6)q^{n-2}+O_d(q^{n-5/2})$ in the case of the upper bound.
\end{Exa}

This paper builds upon the author's 
earlier work \cite{Slavov_FFA} and is inspired by T.~Tao's discussion \cite{Tao_blog} of the Lang--Weil bound through random sampling and the idea of Cafure--Matera \cite{Cafure_Matera} to slice $X$ with planes. A {\it plane} is a $2$-dimensional affine linear subvariety of $\A^n_{\F_q}$. If $H\subset\A^n_{\F_q}$ is any plane, then $\#(X\cap H)(\F_q)$ is either $q^2$, $0$, or $\approx kq$, where $k$ is the number of geometrically irreducible $\F_q$-irreducible components of $X\cap H$. For $0\leq k\leq d$, we exhibit a small interval $I_k=[a_k,b_k]$ containing $kq$ so that if we also define 
$I_\infty=\{q^2\}$, then each $\#(X\cap H)(\F_q)$ belongs to $\bigcup I_k$. 

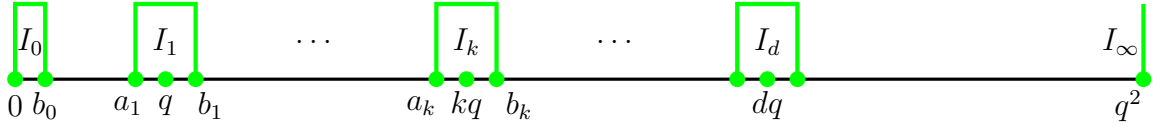
\bigskip

\hspace*{-0.5cm}\begin{tikzpicture}
\draw[very thick](0,0)--(15,0);
\fill[green](0,0) circle (3pt);
\fill[green](2,0) circle (3pt);
\fill[green](6,0) circle (3pt);
\fill[green](10,0) circle (3pt);
\fill[green](15,0) circle (3pt);

\draw[ultra thick, green] (15,0)--(15,1);
\draw[ultra thick, green](0,0)--(0,1)--(0.4,1)--(0.4,0);
\draw[ultra thick, green](1.6,0)--(1.6,1)--(2.4,1)--(2.4,0);
\draw[ultra thick, green](5.6,0)--(5.6,1)--(6.4,1)--(6.4,0);
\draw[ultra thick, green](9.6,0)--(9.6,1)--(10.4,1)--(10.4,0);

\fill[green](0.4,0) circle (3pt);
\fill[green](1.6,0) circle (3pt);
\fill[green](2.4,0) circle (3pt);
\fill[green](5.6,0) circle (3pt);
\fill[green](6.4,0) circle (3pt);
\fill[green](9.6,0) circle (3pt);
\fill[green](10.4,0) circle (3pt);

\node at (1.5,-0.35){$a_1$};
\node at (2,-0.35){$q$};
\node at (2.6,-0.35){$b_1$};

\node at (5.4,-0.35){$a_k$};
\node at (6.7,-0.35){$b_k$};

\node at (0.4,-0.35){$b_0$};
\node at (0,-0.35){$0$};
\node at (6,-0.35){$kq$};
\node at (10,-0.35){$dq$};
\node at (14.8,-0.35){$q^2$};

\node at (0.2,0.5) {$I_0$};
\node at (2,0.5) {$I_1$};
\node at (4,0.5){$\dots$};
\node at (6,0.5) {$I_k$};
\node at (8,0.5) {$\dots$};
\node at (10,0.5) {$I_d$};
\node at (14.7,0.5) {$I_\infty$};
\end{tikzpicture}

\bigskip

The problem when it comes to the upper bound is that when $k$ is large, planes $H$ with
$\#(X\cap H)(\F_q)\in I_k$ contribute significantly towards the count $\#X(\F_q)$. However, it turns out that 
the number of such $H$'s decreases quickly as $k$ grows. 

\section{A collection of small intervals}

\begin{Lem}[\cite{Schmidt}, Lemma 5] 
Let $C\subset\A^2_{\F_q}$ be a curve of degree $d$. Let $k$ be the number of geometrically irreducible $\F_q$-irreducible components of $C$. Then
\[|\#C(\F_q)-kq|\leq
(d-1)(d-2)\sqrt{q}+d^2+d+1.\]
\label{Lem:bounds_near_kq}
\end{Lem}

It will be crucial to give a refined upper bound when $k=1$.

\begin{Lem}
Let $C\subset\A^2_{\F_q}$ be a curve of degree $d$. 
Suppose that $C$ has exactly one geometrically irreducible $\F_q$-irreducible component. Then
\[|C(\F_q)|\leq q+(d-1)(d-2)\sqrt{q}+1.\]
\label{Lem:important}
\end{Lem}

\begin{proof}
Let $C_1,\dots,C_s$ be the $\F_q$-irreducible components of 
$C$. Suppose that $C_1$ is geometrically irreducible, but $C_i$ is not for $i\geq 2$. Let $e=\deg(C_1)$. Note that $(d,e)\neq (2,1)$. 

Using the Aubry--Perret bound \eqref{eq:Aubry_Perret} for $C_1$ and Lemma 2.3 in \cite{Cafure_Matera} for each $C_i$ with $i\geq 2$, we estimate
\begin{align*}
|C(\F_q)| 
&\leq |C_1(\F_q)|+\sum_{i=2}^s |C_i(\F_q)|\\
&\leq q+(e-1)(e-2)\sqrt{q}+1+\sum_{i=2}^s (\deg C_i)^2/4\\
&\leq q+(e-1)(e-2)\sqrt{q}+1+(d-e)^2/4\\
&\leq q+(d-1)(d-2)\sqrt{q}+1;
\end{align*}
to justify the last inequality in the chain, note that it is
 equivalent to
\[(d-e)\left((d+e-3)\sqrt{q}-\frac{d-e}{4}\right)\geq 0\]
and holds true because either $e=d$, or else $d-e>0$ and we
can write
\[(d+e-3)\sqrt{q}-\frac{d-e}{4}
\geq (d+e-3)\sqrt{2}-\frac{d-e}{4}
\geq \frac{(4\sqrt{2}-1)d+(4\sqrt{2}+1)e-12\sqrt{2}}{4}>0
\]
(using that $e\geq 1$ and $d\geq 3$ on the last step).
\end{proof}

Let $a_0=0$, $b_0=d^2/4$, $a_1=q-(d-1)(d-2)\sqrt{q}-d+1$,
$b_1=q+(d-1)(d-2)\sqrt{q}+1$.
For $2\leq k\leq d$, set
$a_k=kq-(d-1)(d-2)\sqrt{q}-d^2-d-1$ and
$b_k=kq+(d-1)(d-2)\sqrt{q}+d^2+d+1$. Finally, set
$a_\infty=b_\infty=q^2$. Define 
$I_k\colonequals [a_k,b_k]$ for 
$k\in\{0,\dots,d\}\cup\{\infty\}$. 

\begin{Lem}
Let $X\subset\A^n_{\F_q}$ be a hypersurface of degree $d$. Let $H\subset\A^n_{\F_q}$ be a plane. Then $\#(X\cap H)(\F_q)\in I_k$ for some $k\in\{0,\dots,d\}\cup\{\infty\}$. 
\end{Lem}

\begin{proof}
If $X\cap H=\emptyset$, then $\#(X\cap H)(\F_q)=0\in I_0$. 
If $H\subset X$, then $X\cap H=H$ and $\#(X\cap H)(\F_q)=q^2\in I_\infty$. Suppose that $X\cap H\neq\emptyset$ and $H\not\subset X$. Let $k$ be the number of geometrically irreducible $\F_q$-irreducible components of the degree $d$ plane curve $X\cap H\subset H\simeq\A^2_{\F_q}$. 
Then $0\leq k\leq d$. 
If $k=0$, the proof of Lemma 11 in \cite{Slavov_FFA} gives
$\#(X\cap H)(\F_q)\leq d^2/4$. If $k=1$, we use Lemma \ref{Lem:important} and the lower bound from \eqref{eq:Aubry_Perret} applied to a geometrically irreducible $\F_q$-irreducible component (necessarily of degree $\leq d$) of $X$. 
For $2\leq k\leq d$, use Lemma \ref{Lem:bounds_near_kq}. 
\end{proof}

Alternatively, one could take $b_d=dq$ by the Schwartz--Zippel lemma. 

When it comes to giving an upper bound for $|X(\F_q)|$, it will be more convenient to work with $J_1\colonequals I_0\cup I_1$ and $J_i\colonequals I_i$ for $i\in\{2,\dots,d\}\cup\{\infty\}$.

\section{Probability estimates}

We spell out in detail the proof of Theorem \ref{thm:pi}; the proofs of the remaining results will then require only slight modifications. The implied constant in each $O$-notation is allowed to depend only on $d$ (a priori, possibly also on $n$), but not on $q$ or $X$.

\begin{proof}[Proof of Theorem \ref{thm:pi}]
Set $N\colonequals |X(\F_q)|$.
For a plane $H\subset\A^n_{\F_q}$ chosen uniformly at random, consider 
$\#(X\cap H)(\F_q)$ as a random variable. Let $\mu$ and $\sigma^2$ denote its mean and variance. Lemma 10 in \cite{Slavov_FFA} and \eqref{eq:Lang_Weil_classical} imply
\begin{equation}
\mu=\frac{N}{q^{n-2}}\quad\text{and}\quad\sigma^2\leq \frac{N}{q^{n-2}}\leq q+O(\sqrt{q}).
\label{eq:var_bound}
\end{equation}

Write
\begin{equation}
\frac{N}{q^{n-2}}=\mu\leq
\sum_{k\in\{1,\dots,d\}\cup\{\infty\}}\Prob\Big(\#(X\cap H)(\F_q)\in J_k\Big)b_k.
\label{eq:sum_of_probabilities}
\end{equation}

For 
$k\in\{1,\dots,d\}\cup\{\infty\}$, denote
\[p_k\colonequals\Prob\Big(\#(X\cap H)(\F_q)\in J_k\Big).\]

We can assume that $q$ is large enough so that the intervals $J_1,\dots,J_d$ are pairwise disjoint. 

Let $k\in\{2,\dots,d\}$. 
If $H$ is a plane such that $\#(X\cap H)(\F_q)\in J_k\cup\dots\cup J_d$, then
\begin{equation}
|\#(X\cap H)(\F_q)-\mu|\geq a_k-\frac{N}{q^{n-2}}\geq (k-1)q-O(\sqrt{q}).
\label{eq:distance_from_mean}
\end{equation}
Define $t$ via 
$(k-1)q-O(\sqrt{q})=t\sigma$; then Chebyshev's inequality and 
the variance bound \eqref{eq:var_bound} imply
\begin{align}
p_k+\cdots+p_d =\Prob\Big(\#(X\cap H)(\F_q)\in J_k\cup\dots\cup J_d\Big) &\leq\frac{1}{t^2}\notag \\
&=\frac{\sigma^2}{((k-1)q-O(\sqrt{q}))^2}\notag \\
&\leq\frac{q+O(\sqrt{q})}{((k-1)q-O(\sqrt{q}))^2}\notag \\
&=\frac{1}{(k-1)^2 q}+O(q^{-3/2}).
\label{eq:key_prob_estimate}
\end{align}

If $H$ is a plane such that $\#(X\cap H)(\F_q)=q^2$, then
\[|\#(X\cap H)(\F_q)-\mu|= q^2-\frac{N}{q^{n-2}}\geq q^2-O(q).\]
Define $t$ via $q^2-O(q)=t\sigma$; then
\[
p_\infty\leq \frac{1}{t^2}=\frac{\sigma^2}{(q^2-O(q))^2}\leq 
\frac{q+O(\sqrt{q})}{(q^2-O(q))^2}=
q^{-3}+O(q^{-7/2}),\quad\text{and hence}\quad p_\infty b_\infty=O(q^{-1}).
\]

Note that $b_k-b_{k-1}=q+O(1)$ for $2\leq k\leq d$. 
We now go back to \eqref{eq:sum_of_probabilities} and apply the Abel summation formula:
\begin{align*}
\frac{N}{q^{n-2}}=\mu&\leq (p_1+\cdots+p_d)b_1+(p_2+\cdots+p_d)(b_2-b_1)+\cdots+p_d(b_d-b_{d-1})+p_\infty b_\infty\\
&\leq b_1+\frac{1}{1^2}+\cdots+\frac{1}{(d-1)^2}+O(q^{-1/2})\\
&\leq q+(d-1)(d-2)\sqrt{q}+1+\pi^2/6+O(q^{-1/2}).
\end{align*}
Multiply both sides by $q^{n-2}$ to arrive at
\eqref{eq:pi}.

Going through all the explicit inequalities with a 
$O$-term, one can compute explicitly a possible value of the constant
implicit in \eqref{eq:pi}. In fact, since there is a choice of $C_d$ in the Lang--Weil bound that depends only on $d$ and not on $n$, a second look at all the inequalities written down in the proof above reveals that the implied constant in \eqref{eq:pi} can likewise be chosen to not depend on $n$. 
\end{proof}

For the rest of the paper, we follow the notation and proof of Theorem \ref{thm:pi}.

\begin{proof}[Proof of Theorem \ref{thm:lower_bound}]
Say that a plane $H$ is ``bad'' if $\#(X\cap H)(\F_q)\in I_0$ and ``good'' otherwise. If $H\subset\A^2_{\F_q}$ is a bad plane, then
\[|\#(X\cap H)(\F_q)-\mu|\geq \frac{N}{q^{n-2}}-\frac{d^2}{4}\geq q-O(\sqrt{q}).\]
By computations similar to the ones in the proof of Theorem \ref{thm:pi}, the probability that a plane is bad is at most 
$q^{-1}+O(q^{-3/2})$. Every good plane contributes at least $a_1$ to the mean. Therefore
\[
\frac{N}{q^{n-2}}=\mu 
\geq (1-q^{-1}-O(q^{-3/2}))(q-(d-1)(d-2)\sqrt{q}-d+1),
\]
\end{proof}
giving \eqref{eq:lower_bound}.

\begin{proof}[Proof of Corollary \ref{Cor:long}]
Modify the proof of Theorem \ref{thm:lower_bound}, but use
the upper bound for $N$ from \eqref{eq:pi} and the lower bound for $N$ from \eqref{eq:lower_bound} respectively for the upper bound on $\sigma^2$ and the lower bound on 
$N/q^{n-2}-d^2/4$. 
\end{proof}

\begin{proof}[Proof of Theorem \ref{thm:projective}]
We now slice with a random plane $H\subset \P^n_{\F_q}$. The mean $\mu$ of $\#(X\cap H)(\F_q)$
is $N\rho_1$, where $N=|X(\F_q)|$ and 
$\rho_1=(q^3-1)/(q^{n+1}-1)$ is the probability that a plane passes through a given point. Let $\rho_2$ be the probability that a plane passes through two distinct given points. Explicitly (in terms of $q$-binomial coefficients), 
$\rho_2=\binom{n-1}{1}_q/\binom{n+1}{3}_q$. One verifies directly that $\rho_2\leq \rho_1^2$ and expresses $\sigma^2$ as in \cite{Tao_blog}:
\[N^2\rho_1^2+\sigma^2=\mu^2+\sigma^2=\mu+N(N-1)\rho_2\leq\mu+N^2\rho_2\]
to deduce $\sigma^2\leq\mu$. 

We can still take $I_0=[0,d^2/4]$. Use the projective version of \eqref{eq:Aubry_Perret} (Corolalry 2.5 in \cite{Aubry_Perret}). Adapt $I_1$ with 
$a_1=q-(d-1)(d-2)\sqrt{q}+1$. Use $I_\infty=\{q^2+q+1\}$. Up to a summand $d$ to account for points at infinity, the remaining $a_k$ and $b_k$ are unchanged. 
 
Proceed as in the proof of Theorems \ref{thm:pi} and \ref{thm:lower_bound}. On the very last step in proving either bound, mutiply by $1/\rho_1$ rather than by $q^{n-2}$ 
and use that $1/\rho_1=q^{n-2}+O(q^{n-5})$. 
\end{proof}

\section{Explicit versions}

\begin{proof}[Proof of Theorem \ref{thm:explicit}]
The statement clearly holds for $d=1$, so assume that 
$d\geq 2$. We will use the explicit Cafure--Matera bound for $N$. Replace the variance bound \eqref{eq:var_bound} by
\[\sigma^2\leq\frac{N}{q^{n-2}}\leq q+(d-1)(d-2)\sqrt{q}+5d^2+d+1\leq (8.44/7.44)q;\]
to verify the last inequality above, we argue as follows. 
For any $c_1>0$ and $c_2>0$, the function 
$q\mapsto q/(c_1\sqrt{q}+c_2)$ is increasing. Therefore
\[\frac{q}{(d-1)(d-2)\sqrt{q}+5d^2+d+1}>\frac{15d^{13/3}}{(d-1)(d-2)\sqrt{15}d^{13/6}+5d^2+d+1}.\]
It remains to check that the function $g(d)$ on the right-hand side above satisfies $g(d)>7.44$ for any integer $d\geq 2$. On the one hand, $g$ grows like $d^{1/6}$ so one easily exhibits a $d_0$ such that $g(d)>7.44$ for $d>d_0$. Then a simple computer calculation checks that $g(d)>7.44$ for 
integers $d\in\{2,\dots,d_0\}$ as well.

In the same way, one readily checks that the intervals $J_1,\dots,J_d$ are pairwise disjoint. 

For $k\in\{2,\dots,d\}$, replace \eqref{eq:distance_from_mean} by
\[a_k-\frac{N}{q^{n-2}}\geq (k-1)q-2(d-1)(d-2)\sqrt{q}-2(3d^2+d+1)\geq (5.45/7.45)(k-1)q;\]
to check the last inequality, one has to consider only $k=2$ and to argue as above. 

For $k\in\{2,\dots,d\}$, \eqref{eq:key_prob_estimate} is now replaced by
\[p_k+\cdots+p_d\leq \frac{(8.44/7.44)q}{\left((5.45/7.45)(k-1)q\right)^2}<\frac{2.12}{(k-1)^2q}.\]

To bound $p_\infty b_\infty$, note that $q>15d^{13/3}>15\times 2^{13/3}>302$, so
\[p_\infty b_\infty\leq\frac{(8.44/7.44)q}{(q^2-(8.44/7.44)q)^2}q^2
=\frac{8.44\times 7.44q}{(7.44q-8.44)^2}<0.01.\]

Since $b_k-b_{k-1}=q$ for $3\leq k\leq d$, but $b_2-b_1=q+d^2+d$, we have to estimate
$(d^2+d)/q<(d^2+d)/15d^{13/3}<0.02$. The Abel summation argument now gives
\[\frac{N}{q^{n-2}}\leq q+(d-1)(d-2)\sqrt{q}+1+2.12(\pi^2/6+0.02)+0.01<q+(d-1)(d-2)\sqrt{q}+5.\qedhere\]
\end{proof}

\begin{proof}[Proof of Theorem \ref{thm:zone}]
Again, assume $d\geq 2$. We can assume that the right-hand side of \eqref{eq:zone_11} is less than the right-hand side of \eqref{eq:N_with_three_halves}; i.e.,
\[
4(d-1)(d-2)\sqrt{q}+2(d^2+d+13)<q.
\]
This inequality implies in particular that the intervals $J_1,\dots,J_d$ are pairwise disjoint. Note that it is equivalent to $q>r(d)^2$, where $r(d)$ is the positive root of the quadratic equation
$x^2-4(d-1)(d-2)x-2(d^2+d+13)=0$. 

Due to \eqref{eq:N_with_three_halves}, now we can use the variance bound $\sigma^2\leq N/q^{n-2}\leq (3/2)q$. Also, \eqref{eq:N_with_three_halves} gives
\[a_k-\frac{N}{q^{n-2}}= kq-(d-1)(d-2)\sqrt{q}-(d^2+d+1)-\frac{N}{q^{n-2}}\geq \frac{k-1}{2}q\]
for $2\leq k\leq d$. Therefore $p_k+\cdots+p_d$ is now bounded by $6/((k-1)^2q)$.

We bound $(d^2+d)/q$ by $(d^2+d)/(r(d))^2<0.16$ for $d\geq 2$. Finally, note that $q>r(2)^2=38$, so $q\geq 41$, and
we can bound $p_\infty b_\infty$ by $6q/(2q-3)^2<0.04$. Therefore
\[\frac{N}{q^{n-2}}\leq q+(d-1)(d-2)\sqrt{q}+1+6(\pi^2/6+0.16)+0.04<q+(d-1)(d-2)\sqrt{q}+12.\qedhere\]
\end{proof}

\begin{proof}[Proof of Theorem \ref{thm:explicit_lower_bound}]
As above, assume that $d\geq 2$. We bound the variance as
\[\sigma^2\leq\frac{N}{q^{n-2}}\leq q+(d-1)(d-2)\sqrt{q}+5d^2+d+1\leq (8.44/7.44)q.\]
Also,
\[\frac{N}{q^{n-2}}-\frac{d^2}{4}\geq q-(d-1)(d-2)\sqrt{q}-21d^2/4-d-1\geq (6.44/7.44)q.\]
From here, we bound the probability that a plane is bad by $1.6/q$. Thus
\[\frac{N}{q^{n-2}}\geq \left(1-\frac{1.6}{q}\right)(q-(d-1)(d-2)\sqrt{q}-d+1)\geq q-(d-1)(d-2)\sqrt{q}-(d+0.6).\qedhere\]
\end{proof}


\end{document}